\documentclass[11pt]{amsart}

\usepackage[T1]{fontenc}
\usepackage[utf8]{inputenc}
\usepackage{amsmath,amssymb,amsfonts,amsthm}
\usepackage[a4paper,margin=30mm]{geometry}
\usepackage{graphicx}
\usepackage{hyperref}
\newcommand{\iddots}{\mathinner{\reflectbox{$\ddots$}}}

\newtheorem{Lemma}{Lemma}
\newtheorem{Proposition}{Proposition}
\newtheorem{Theorem}{Theorem}
\newtheorem{Corollary}{Corollary}
\theoremstyle{definition}
\newtheorem{Example}{Example}
\theoremstyle{remark}
\newtheorem{Remark}{Remark}

\title[Minimal binomial generators for certain monomial curves]{Minimal Systems of Binomial Generators for the Ideals of Certain Monomial Curves}

\author{Manuel B. Branco}
\address{Departamento de Matem\'aticas, Universidade de \'Evora, 7000-671 \'Evora, Portugal}
\email{mbb@uevora.pt}

\author{Isabel Cola\c{c}o}
\address{Departamento de Matem\'atica e Ci\^encias F\'{\i}sicas, Instituto Polit\'ecnico de Beja, 7800-295 Beja, Portugal}
\email{isabel.colaco@ipbeja.pt}

\author{Ignacio Ojeda}
\address{Departamento de Matem\'aticas, Universidad de Extremadura, 06071 Badajoz, Spain}
\email{ojedamc@unex.es}


\begin{document}

\begin{abstract}
Let $a, b$ and $n > 1$ be three positive integers such that $a$ and $\sum_{j=0}^{n-1} b^j$ are relatively prime. In this paper, we prove that the toric ideal $I$ associated to the submonoid of $\mathbb{N}$ generated by $\{\sum_{j=0}^{n-1} b^j\} \cup \{\sum_{j=0}^{n-1} b^j + a\, \sum_{j=0}^{i-2} b^j \mid i = 2, \ldots, n\}$ is determinantal. Moreover, we prove that for $n > 3$, the ideal $I$ has a unique minimal system of generators if and only if $a < b-1$.
\end{abstract}

\maketitle

\medskip
\noindent
\textcopyright{} 2021 The Authors. This manuscript is available under the
\href{https://creativecommons.org/licenses/by/4.0/}{CC BY 4.0 license}.

\noindent
Published version:
\href{https://doi.org/10.3390/math9243204}{doi:10.3390/math9243204}.
\medskip

\section{Introduction}

Let $\Bbbk$ be a field and let $\mathcal{A} = \{a_1, \ldots, a_n\}$ be a set of positive integers. It is well known that the kernel of the $\Bbbk-$algebra homomorphism \begin{equation}\label{ecu2} \varphi_{\mathcal{A}} : \Bbbk[x_1, \ldots, x_n] \to \Bbbk[t^{a_1}, \ldots, t^{a_n}];\ x_i \mapsto t^{a_i},\ i = 1, \ldots, n,\end{equation} where $x_1, \ldots, x_n$ and $t$ are indeterminates, is a binomial ideal (see \cite{Herzog70}, or \cite{MO18} for a more recent reference). Clearly, $\ker(\varphi_\mathcal{A})$ is the defining ideal of a monomial curve.

Let $b$ be a positive integer and set $r_b(\ell)$ for the $\ell-$th repunit number in base $b$, that is, \[r_b(\ell) = \sum_{j=0}^{\ell-1} b^j.\] By convention, $r_b(0) = 0$. 

The main result in this paper is the explicit determination of a minimal system of binomial generators of $I := \ker \varphi_{\mathcal{A}}$ for \[\mathcal{A} = \{a_i := r_b(n) + a\, r_b(i-1)\ \mid\ i = 1, \ldots, n\},\] where $a$ and $n > 1$ are positive integers. We prove that $I$ is minimally generated by the $2 \times 2$ minors of the matrix \begin{equation}\label{ecu1}
X:= \left(
\begin{array}{cccc} 
x_1^b  & \ldots & x_{n-1}^b & x_n^b \smallskip \\ x_2 & \ldots & x_n & x_1^{a+1}
\end{array}
\right), 
\end{equation}
provided that $\gcd(a_1, \ldots, a_n) = \gcd(a, r_b(n))$ is equal to $1$. In this case, as an immediate consequence, we have that the so-called binomial arithmetical rank of $I$ (see, e.g.,~\cite{Ka17}) is equal to $\binom{n}{2}$. 

Furthermore, we obtain that the $2 \times 2-$minors of $X$ form a minimal Gr\"obner basis with respect to a family of $\mathcal{A}$-graded reverse lexicographical term orders on $\Bbbk[x_1, \ldots, x_n]$ (Theorem \ref{Prop GB2}) and, applying~(\cite{OV2}, Corollary 14), we conclude that for $n > 3$, the ideal $I$ has a unique minimal system of generators if and only if and $a < b-1$ (Corollary \ref{Cor GB2}).

The submonoids of $\mathbb{N}$ generated by $\mathcal{A}$ are studied in detail in~\cite{BCO1} as a generalization of the numerical semigroups introduced by D. Torr\~ao et al. (see~\cite{RBT16,RBT17}); in this context, Corollary \ref{Cor GB2} provides a minimal presentation of the submonoid of $\mathbb{N}$ generated by $a_1, \ldots, a_n$, providing an original result not considered in Torr\~ao's PhD thesis.

To achieve our main result (Theorem \ref{Prop GB2}), we first compute the ideal $J$ of the projective monomial curve defined by the kernel of the $\Bbbk-$algebra homomorphism \begin{equation}\label{ecu3}\Bbbk[x_1, \ldots, x_n] \to \Bbbk[t^{r_b(0)}\, s, \ldots, t^{r_b(n-1)}\, s];\quad x_i \mapsto t^{r_b(i-1)}\, s,\ i = 1, \ldots, n,\end{equation} where $s$ is also an indeterminate. This intermediate result (Proposition \ref{Prop GB1}) has its own interest, as it exhibits another family of semigroup ideals that are determinantal and have unique minimal system of binomial generators (Corollary \ref{Cor GB1}).

Throughout the paper, we keep the notation established in this introduction. Moreover, as the case $ n = 2 $ is trivial and the case $ n = 3 $ is well known for any $a_1, a_2$ and $a_3$ (see~\cite{Herzog70}), we suppose that $n > 3$ whenever necessary.

The explicit description of minimal systems of binomial generators of monomial curves, and in a broader context of toric ideals, is a long-established research topic since J. Herzog, in his celebrated paper~\cite{Herzog70}, characterized the minimal systems of binomial generators of (all) the monomial curves in affine three-dimensional space. The elegance of Herzog's result for the three-dimensional case contrasts with the fact that no explicit description is known for the general case. Particular advances are just known for low-dimensional cases (see, e.g.,~\cite{Bresinksy} or more recently in~\cite{KaOj} and the references therein) or for special families of monomial curves as presented in this paper; due to its proximity to the present work, we highlight the article by D.P. Patil~\cite{patil} as one among many others. 

We finally emphasize that, despite of not being the aim this paper, the study of the defining ideal of monomials curves have its own interest for applications to other areas such as linear programming (see, e.g.,~\cite{villarreal}), coding theory (see, e.g.,~\cite{codes} or algebraic statistics, where the minimal systems of bionomial generators are called Markov bases and the uniqueness property has special consideration (see~\cite{markov}).

%%%%%%%%%%%%%%%%%%%%%%%%%%%%%%%%%%%%%%%%%%%%%%%%%%%%%%%%%%%
%%%%%%%%%%%%%%%%%%%%%%%%%%%%%%%%%%%%%%%%%%%%%%%%%%%%%%%%%%%
\section{Preliminaries}\label{S2}

Let $a, b$ and $n$ be three positive integers such that $n > 3$. Consider the sequence of positive integers $(a_i)_{i \geq 1}$ such that \[a_i := r_b(n) + a\, r_b(i-1),\] for every $i \geq 1$. 

In this section, we present several lemmas that reflect the arithmetic structure of the sequence $(a_i)_{i \geq 1}$. In addition, we present the family of term orders that will be used throughout the paper. 

\begin{Lemma}\label{lema1}
The following equality holds: $a_{n+k} = a_k + a\, b^{k-1}\, a_1,$ for all $k \geq 1$. In particular, $a_{n+1} = (1+a)\, r_b(n)$.
\end{Lemma}

\begin{proof}
It suffices to observe that $r_b(n+k-1) = r_b(k-1) + b^{k-1}\, r_b(n)$, for all $k \geq 1$, and, consequently, that $a_{n+k} = r_b(n) + + a\, r_b(n+k-1) = a_1 + a\, r_b(n+k-1) = a_1 + a\, ( r_b(k-1) + b^{k-1} r_b(n)) = a_k + a\, b^{k-1}\, a_1$, for all $k \geq 1$. Finally, as $a_1 = r_b(n)$, the last statement is straightforward 
\end{proof}

Notice that, by Lemma \ref{lema1}, the set $\mathcal{A} = \{a_1, \ldots, a_n\}$ is a system of generators of the submonoid of $\mathbb{N}$ generated by the sequence $(a_i)_{i \geq 1}.$

\begin{Lemma}\label{lema2}
For each pair of positive integers $j$ and $k$, it holds that \[b\, a_j + a_{j+k} = b\, a_{j+k-1} + a_{j+1}.\]
\end{Lemma}

\begin{proof}
As $a_{j+k} = a_1 + a\, r_b(j+k-1) = a_1 + a\, (r_b(j-1) +  b^{j-1}\, r_b(k)) = a_j + a\, b^{j-1}\, r_b(k)$, we conclude that 
\begin{align*}
b\, a_j + a_{j+k} & = b\, a_j + a_j + a\, b^{j-1}\, r_b(k)= \\ & = b\, a_j + a_j + a\, b^ {j-1}\, (b\, r_b(k-1)+1) = \\ & = b\, (a_j + a\, b^{j-1}\, r_b(k-1)) + a_j + a\, b^{j-1} = \\ & =  b\, a_{j+k-1} + (a_1 + a\, r_b(j-1)) + a\, b^{j-1} = \\ & =  b\, a_{j+k-1} + a_{j+1},
\end{align*}
as claimed.
\end{proof}

Let $\prec_i$ be the term order on $\Bbbk[x_1, \ldots, x_n]$ defined by the following matrix
\[
M:=\left(
\begin{array}{ccc|c|ccc}
a_1 & \ldots & a_i & a_{i+1} & a_{i+2} & \ldots & a_n\\ \cline{1-7}
0 &  & -1 & 0 & 0 & \ldots & 0 \\ 
& \iddots &  & \vdots & \vdots & & \vdots \\ 
-1 &  & 0 & 0 & 0 & \ldots & 0 \\ \cline{1-7}
0 & \ldots &  0 & 0 & 0 &  & -1 \\ 
\vdots & & \vdots & \vdots & & \iddots & \\ 
0 & & 0 & 0 & -1 & & 0
\end{array}
\right).
\]
We observe that $\prec_i$ is the $\mathcal{A}$-graded reverse lexicographical term order on $\Bbbk[x_1, \ldots, x_n]$ induced by $x_i \prec_i x_{i-1} \prec_i \ldots \prec_i x_1 \prec_i x_n \prec_i \ldots \prec_i x_{i+1};$ in particular, $x_i$ is the smallest variable for $\prec_i$.

\begin{Lemma}\label{lema3}
If $j \in \{1, \ldots, n-2\}$ and $k \in \{j+1, \ldots, n-1\}$, then \[x_j^b x_{k+1} \prec_i x_{j+1} x_k^b\] if and only if  $i \leq j$ or $k+1 \leq i$.
\end{Lemma}

\begin{proof}
By Lemma \ref{lema2}, $b\, a_j + a_{k+1} = a_{j+1} + b\, a_k$, so we just need to decide what the variable $x_j, x_{j+1}, x_k$ or $x_{k+1}$ is cheapest for the order defined by the last $n-1$ rows of $M$. As $j < j+1 \leq k < k+1$, according to the definition of $\prec_i$, the variable $x_{k+1}$ is cheaper than the other three when $j \leq i$ or $k+1 \leq i$; thus, $x_j^b x_{k+1} \prec_i x_{j+1} x_k^b$ in these cases. Conversely, if $j+1 \leq i \leq k$, then either $x_k$ or $x_{j+1}$ is cheaper than the others if $k=i$ or $k \neq i$, respectively. Therefore $x_j^b x_{k+1} \succ_i x_{j+1} x_k^b$ when $j+1 \leq i \leq k$, and we are done.
\end{proof}

%%%%%%%%%%%%%%%%%%%%%%%%%%%%%%%%%%%%%%%%%%%%%%%%%%%%%%%%%%%
%%%%%%%%%%%%%%%%%%%%%%%%%%%%%%%%%%%%%%%%%%%%%%%%%%%%%%%%%%%
\section{Gr\"obner Bases and Minimal Generators for \boldmath$J$}\label{S3}

We keep the notation of the Introduction and Section \ref{S2}.  

Let $I_2(Y)$ be the ideal of $\Bbbk[x_1, \ldots, x_n]$ generated by the $2 \times 2-$minors of \[
Y:= \left(
\begin{array}{ccccc} 
x_1^b & x_2^b & \ldots & x_{n-1}^b \smallskip \\ x_2 & x_3 & \ldots & x_n
\end{array}
\right).
\]
Let $\mathcal{G}^{(i)}_1, \mathcal{G}^{(i)}_2$ and $\mathcal{G}^{(i)}_3$ be defined as follows:
\[
\mathcal{G}^{(i)}_1 = \left\{ \underline{x_{j+1} x_k^b} - x_j^b x_{k+1} \ \mid\ j \in \{i, \ldots, n-2\},\ k \in \{j+1,\ldots, n-1\}\right\},
\]
\[
\mathcal{G}^{(i)}_2 = \left\{ \underline{x_{j+1} x_k^b} - x_j^b x_{k+1} \ \mid\ j \in \{1, \ldots, i-2\},\ k \in \{j+1,\ldots, i-1\}\right\},
\]
\[
\mathcal{G}^{(i)}_3 = \left\{ \underline{x_j^b x_{k+1}} - x_{j+1} x_k^b \ \mid\ j \in \{1, \ldots, i-1\},\ k \in \{i,\ldots, n-1\}\right\}
\]
and let $\mathcal{G}^{(i)}_Y$ be equal to $\mathcal{G}^{(i)}_1 \cup \mathcal{G}^{(i)}_2 \cup \mathcal{G}^{(i)}_3$.

Notice that, by Lemma \ref{lema3}, the underlined monomials are the leading terms with respect to $\prec_i$ of the corresponding binomials.

\begin{Proposition}\label{Prop GB1}
With the above notation, the set $\mathcal{G}^{(i)}_Y$ is the reduced Gr\"obner basis of $I_2(Y)$ with respect to $\prec_i$. In particular, the cardinality of $\mathcal{G}^{(i)}_Y$ is $\binom{n-1}2$.
\end{Proposition}

\begin{proof}
First, let us see that $\mathcal{G}^{(i)}_Y$ is a Gr\"obner basis. By the Buchberger’s Criterion (see, e.g.,~\cite{MS05}, Theorem 3.3), it suffices to verify that each S-pair of elements in $\mathcal{G}^{(i)}_Y$ can be reduced to zero by $\mathcal{G}^{(i)}_Y$ using the division algorithm. To do this, we distinguish \mbox{several cases:}
\begin{itemize}
\item Let $f \in \mathcal{G}^{(i)}_1$, that is to say, $f = x_{j+1} x_k^b - x_j^b x_{k+1}$, for some $j \in \{i, \ldots, n-2\}$ and $k \in \{j+1,\ldots, n-1\}.$
\begin{itemize}
\item[$\circ$] Let $g = x_{l+1} x_m^b - x_l^b x_{m+1} \in \mathcal{G}^{(i)}_1$. If $\gcd(x_{j+1} x_k^b, x_{l+1} x_m^b) = 1$, then $S(f,g)$ reduces to zero with respect to $\{f,g\} \subset \mathcal{G}^{(i)}_Y$. Otherwise, $j = l, j+1 = m, k = l+1$ or $k = m$. If \framebox{$j=l$} then $S(f,g) = x_m^b (- x_j^b x_{k+1}) - x_k^b (- x_j^b x_{m+1}) = x_j^b (x_k^b x_{m+1} - x_m^b x_{k+1})$ reduces to zero with respect to $\mathcal{G}^{(i)}_Y$. If \framebox{$j+1 = m$}, then \[S(f,g) = x_{l+1} x_{j+1}^{b-1}(- x_j^b x_{k+1}) - x_k^b(-x_l^b x_{j+2}).\] Now, as $i \leq j < j+1 \leq k < k + 1$ and $i \leq l < l+1 \leq m=j+1 < j+2$, the leading term of $S(f,g)$ with respect to $\prec_i$ is $x_k^b x_l^b x_{j+2}$. Then $S(f,g) = x_l^b ( \underline{x_k^b x_{j+2}} - x_{j+1}^b x_{k+1}) + x_{j+1}^{b-1} x_{k+1}  (x_l^b x_{j+1} - x_{l+1} x_j^b)$ reduces to zero with respect to $\mathcal{G}^{(i)}_Y$. By symmetry, the case \framebox{$k=l+1$} is completely similar to the latter one.  Finally, if \framebox{$k=m$}, then $S(f,g) = -x_{l+1}(-x_j^b x_{k+1}) - x_{j+1}(-x_l^b x_{k+1}) =  x_{k+1}(x_j^b x_{l+1} - x_{j+1} x_l^b)$ reduces to zero with respect to $\mathcal{G}^{(i)}_Y$.
\item[$\circ$] Let $g = x_{l+1} x_m^b - x_l^b x_{m+1} \in \mathcal{G}^{(i)}_2$. If $\gcd(x_{j+1} x_k^b, x_{l+1} x_m^b) = 1$, then $S(f,g)$ reduces to zero with respect to $\{f,g\} \subset \mathcal{G}^{(i)}_Y$. Otherwise, $j = l, j+1 = m, k = l+1$ or $k = m$. First, we observe that the cases $j=l$ and $k=m$ produce the same S-polynomial as in the corresponding case for $g \in \mathcal{G}^{(i)}_1$; so, we just focus on the cases $j+1 =m$ and $k=l+1$. If \framebox{$j+1 =m$}, then $i \leq j < j+1 \leq k < k + 1$ and $l < l+1 \leq m=j+1 < j+2 = m-1 \leq i$, therefore $i < j+1 = m < i$, a contradiction. Finally, if \framebox{$k=l+1$}, then $i \leq j < j+1 \leq k < k + 1$ and $l < l+1=k \leq m < m+1 \leq i$, so $i < k = l+1 < i$, a contradiction again.
\item[$\circ$] Let $g = x_l^b x_{m+1} - x_{l+1} x_m^b \in \mathcal{G}^{(i)}_3$. If $\gcd(x_{j+1} x_k^b, x_l^b x_{m+1}) = 1$, then $S(f,g)$ reduces to zero with respect to $\{f,g\} \subset \mathcal{G}^{(i)}_Y$. Otherwise, $j+1 = l, j = m, k = l$ or $k = m+1$. If \framebox{$j+1 = l$}, then $i \leq j < j+1 = l \leq k < k + 1$ and $l \leq i-1$; so $i < j+1 = l \leq i-1$, a contradiction. If \framebox{$j = m$} (or \framebox{$k = l$}, respectively) then $S(f,g) = x_m^b(x_k^b x_{l+1} - x_l^b x_{k+1})$ (or $S(f,g) = x_{k+1}(x_{j+1} x_m^b - x_j^b x_{m+1}),$  respectively) reduces to zero with respect to $\mathcal{G}^{(i)}_Y$. Finally, if \framebox{$k=m+1$}, then \[S(f,g) = x_l^b(- x_j^b x_{m+2}) - x_{j+1} x_{m+1}^{b-1} (- x_{l+1} x_{m}^b).\] Now, as $i \leq j < j+1 \leq k = m+1 < k + 1$, $l \leq i-1$ and $i \leq m$, then $x_{l+1}$ or $x_{m-1}$ is cheaper than the others for the order induced by the last $n-1$ rows of the matrix $M$, therefore leading term of $S(f,g)$ is $x_l^b x_j^b x_{k+1}$ and thus, $S(f,g) = - x_j^b (\underline{x_l^b x_{m+2}} - x_{l+1} x_{m+1}^b) - x_{l+1} x_{m+1}^{b-1}(x_j^b x_{m+1}- x_{j+1} x_{m}^b)$ reduces to zero with respect to $\mathcal{G}^{(i)}_Y$.
\end{itemize}
\item Let $f \in \mathcal{G}^{(i)}_2$, that is to say, $f = x_{j+1} x_k^b - x_j^b x_{k+1}$, for some $j \in \{1, \ldots, i-2\}$ and $k \in \{j+1,\ldots, i-1\}.$
\begin{itemize}
\item[$\circ$] Let $g = x_{l+1} x_m^b - x_l^b x_{m+1} \in \mathcal{G}^{(i)}_2$. If $\gcd(x_{j+1} x_k^b, x_{l+1} x_m^b) = 1$, then $S(f,g)$ reduces to zero with respect to $\{f,g\} \subset \mathcal{G}^{(i)}_Y$. Otherwise, $j = l, j+1 = m, k = l+1$ or $k = m$. If \framebox{$j=l$} (or \framebox{$k = m$}, respectively), then $S(f,g) = x_j^b(x_k^bx_{m+1}-x_{k+1}x_m^b)$ (or $S(f,g) = x_{k+1}(x_{l+1}  x_j^b- x_l^bx_{j+1})$, respectively) reduces to zero with respect to $\mathcal{G}^{(i)}_Y$. If \framebox{$j+1 = m$}, then \[S(f,g) = x_{l+1} x_m^{b-1}(- x_{m-1}^b x_{k+1}) - x_k^b (- x_l^b x_{m+1})\] and, as $l+1 \leq m=j+1 \leq k \leq i-1$, the leading term of $S(f,g)$ is equal to $ x_{l+1} x_m^{b-1} x_{m-1}^b x_{k+1}$. Thus, $S(f,g) = -x_m^{b-1} x_{k+1} (\underline{x_{l+1} x_{m-1}^b} - x_l^b x_m) + x_l^b(x_k^b x_{m+1} - x_m^{b} x_{k+1})$ reduces to zero with respect to $\mathcal{G}^{(i)}_Y$; observe that $l < l+1 \leq m$ implies that the leading term of $x_{l+1} x_{m-1}^b - x_l^b x_m$ is actually $x_{l+1} x_{m-1}^b$. Finally, by symmetry, the case \framebox{$k=l+1$} is completely similar to the latter one.
\item[$\circ$] Let $g = x_l^b x_{m+1} - x_{l+1} x_m^b \in \mathcal{G}^{(i)}_3$. If $\gcd(x_{j+1} x_k^b, x_l^b x_{m+1}) = 1$, then $S(f,g)$ reduces to zero with respect to $\{f,g\} \subset \mathcal{G}^{(i)}_Y$. Otherwise, $j+1 = l, j = m, k = l$ or $k = m+1$. If \framebox{$j+1=l$}, then \[S(f,g) = x_l^{b-1} x_{m+1} (-x_{l-1}^b x_{k+1}) -x_k^b ( - x_{l+1} x_m^b).\] Furthermore, as $l = j+1 \leq k \leq i-1$ and $l \leq i-1 < i \leq m < m+1$, we have that the leading term is $x_l^{b-1} x_{m+1} x_{l-1}^b x_{k+1}$ if $k=l$ and $x_k^b x_{l+1} x_m^b$ otherwise. In ther first case, $S(f,g) =  x_{m+1}x_{k-1}^b +x_k x_m^b$ reduces to zero with respecto to $\mathcal{G}^{(i)}_Y$. In the second case, $S(f,g) = x_m^b (\underline{x_k^b x_{l+1}} - x_{k+1} x_l^b) +  x_l^{b-1}  x_{k+1} (x_m^b x_l -   x_{m+1} x_{l-1}^b)$ reduces to zero with respect to $\mathcal{G}^{(i)}_Y$. If \framebox{$j = m$} (or \framebox{$k = l$}, respectively) then $S(f,g) = x_m^b(x_k^b x_{l+1} - x_l^b x_{k+1})$ (or $S(f,g) = x_{k+1}(x_{j+1} x_m^b - x_j^b x_{m+1}),$  respectively) reduces to zero with respect to $\mathcal{G}^{(i)}_Y$. Finally, if \framebox{$k=m+1$}, then 
$j+1 \leq k = m+1 \leq i-1, l \leq i-1$ and $i \leq m$; so, $m+1 < i \leq m$, a contradiction.
\end{itemize}
\item Let $f \in \mathcal{G}^{(i)}_3$, that is to say, $f = x_j^b x_{k+1} - x_{j+1} x_k^b,$ for some $j \in \{1, \ldots, i-1\}$ and $k \in \{i,\ldots, n-1\}$.
\begin{itemize}
\item[$\circ$] Let $g = x_l^b x_{m+1} - x_{l+1} x_m^b \in \mathcal{G}^{(i)}_3$. If $\gcd(x_j^b x_{k+1}, x_l^b x_{m+1}) = 1$, then $S(f,g)$ reduces to zero with respect to $\{f,g\} \subset \mathcal{G}^{(i)}_Y$. Otherwise, $j = l, j = m+1, k+1 = l$ or $k = m$. As $j \leq i-1 < i \leq k$ and $l \leq i-1 < i \leq m$, the cases \framebox{$j = m+1$} and \framebox{$k+1 = m$} cannot occur. If \framebox{$j = l$} (or \framebox{$k=m$}, respectively), then $S(f,g) = x_{l+1}(x_{k+1} x_m^b - x_{m+1}x_k^b)$ (or $S(f,g) = x_m^b (x_l^b x_{j+1} - x_j^b x_{l+1} )$, respectively) reduces to zero with respect $\mathcal{G}^{(i)}_Y$.
\end{itemize}
\end{itemize}

{Once} we know that $\mathcal{G}^{(i)}_Y$ is Gr\"obner basis, it is immediate to see that it is reduced since the leading term of $f \in \mathcal{G}^{(i)}_Y$ does not divide any other monomial that appears in a binomial of $\mathcal{G}^{(i)}_Y \setminus \{f\}$.

It remains to prove that $\mathcal{G}^{(i)}_Y$ generates $I_2(Y)$. Clearly, $\mathcal{G}^{(i)}_Y$ is contained in the set of $2 \times 2-$minors of $Y$. Moreover, as the cardinality of $\mathcal{G}^{(i)}_1, \mathcal{G}^{(i)}_2$ and $\mathcal{G}^{(i)}_3$ are 
\[\Big((n-1)-i\Big) + \Big((n-1)-i-1\Big) + \ldots + 1 = \binom{n-i}2,\]
\[\Big((i-1)-1\Big) + \Big((i-1)-2 \Big) + \ldots + 1 = \binom{i-1}2\]
and
\[(i-1)(n-i),\]
respectively, we have that the cardinality of $\mathcal{G}^{(i)}_Y$ is equal to $\binom{n-1}2$ which is the number of $2 \times 2-$minors of $Y$. Therefore, $\mathcal{G}^{(i)}_Y$ generates $I_2(Y)$ and we are done.
\end{proof}

\begin{Example}\label{Ex5}
We observe that the reduced Gr\"obner basis, $\mathcal{G}^{(i)}_Y$, of $I_2(Y)$ with respect to $\prec_i$ is not an universal Gr\"obner basis. For example, if $n=b=5$ and $\prec$ is the term order defined by 
\[
\left(
\begin{array}{ccccc}
a_1 & a_2 & a_3 & a_4 & a_5 \\
0 & 0 & -1 & 0 & 0 \\
0 & 0 & 0 & 0 & -1 \\
0 & 0 & 0 & -1 & 0 \\
0 & -1 & 0 & 0 & 0
\end{array}
\right),
\]
then one can check (using, for example, Singular~\cite{singular}) that the reduced Gr\"obner basis of the ideal $I_2(Y)$ with respect to $\prec$ has eight generators; however, $\mathcal{G}^{(i)}_Y$ contains $\binom{5-1}2 = 6$ binomials only. 

Alternatively, one can see that $\mathcal{G}^{(i)}_Y$ is not an universal Gr\"obner basis of $I_2(Y)$ by using~\mbox{(\cite{BR15}, Theorem 4.1).}
\end{Example}

We now consider the $2 \times n-$integer matrix $B$ whose $j-$th column is \[\mathbf{a}_j := \left(\begin{array}{c} r_b(j-1)\\ 1 \end{array}\right),\quad j = 1, \ldots, n.\] 

\begin{Remark}
Observe that $a_j = \big(a, r_b(n)\big) \cdot \mathbf{a}_j,$ for every $j = 1, \ldots, n.$
\end{Remark}

Notice that the semigroup ideal associated to $\{\mathbf{a}_1, \ldots, \mathbf{a}_n\}$ is equal to $J$; indeed, $J$ is the kernel of \eqref{ecu3}.

\begin{Corollary}\label{Cor GB1}
The ideal $J$ is minimally generated by the $2 \times 2-$minors of $Y$. Moreover, $J$ has a unique minimal system of binomial generators.
\end{Corollary}

\begin{proof}
Let $I_2(Y)$ the ideal generated by the $2 \times 2-$minors of $Y$. Since $b\, \mathbf{a}_j + \mathbf{a}_{k+1} = \mathbf{a}_{j+1} + b\, \mathbf{a}_k$ for every $j$ and $k$, we have that $I_2(Y) \subseteq J$. 

Conversely, let $C$ be the $(n-2) \times n-$matrix 
\[
\left(
\begin{array}{ccccccccccc}
b & -1 & -b &  1 & 0 & \ldots & 0 & 0 & 0 & 0 \\
0 &  b & -1 & -b & 1 & \ldots & 0 & 0 & 0 & 0 \\
0 & 0 &  b & -1 & -b & \ldots & 0 & 0 & 0 & 0 \\
0 & 0 & 0 &  b & -1 & \ldots & 0 & 0 & 0 & 0 \\
0 & 0 & 0 & 0 & b  & \ldots & 0 & 0 & 0 & 0 \\
\vdots & \vdots & \vdots & \vdots & \vdots & \ddots & \vdots & \vdots & \vdots & \vdots \\
0 & 0 & 0 & 0 & 0 & \ldots & b & -1 & -b & 1\\
0 & 0 & 0 & 0 & 0 & \ldots & 0 & b & -(b+1) & 1\\
\end{array}
\right)
\]
and let $I_{C}$ be the ideal of $\Bbbk[x_1, \ldots, x_n]$ generated by \[\{ \mathbf{x}^{\mathbf{u}_+} - \mathbf{x}^{\mathbf{u}_-}\ \mid\ \mathbf{u}\ \text{is a row of}\ {C} \},\] where $\mathbf{u}_+$ and $\mathbf{u}_-$ denote the positive and negative parts of $\mathbf{u}$, respectively. Clearly, $I_{C} \subseteq I_2(Y)$.

Now, as the determinant of the submatrix of $C$ consisting in the last $n-2$ columns is $1$, the rows of $C$ generates a rank $n-2$ subgroup $G_{C}$ of $\mathbb{Z}^n$ such that $\mathbb{Z}^n/G_{C}$ is torsion free. Moreover, as $B\, C^\top = 0$, we conclude that the rows of $C$ generate $\ker_\mathbb{Z}(B)$. Therefore, by~(\cite{MS05}, Lemma 7.6), \[J = I_{C}: \Big(\prod_j x_j \Big)^\infty  \subseteq I_2(Y) : \Big(\prod_j x_j \Big)^\infty .\] By Proposition \ref{Prop GB1} and~(\cite{BSR99}, Theorem 3.1), we have that $I_2(Y):x_i^\infty = I_2(Y)$ for every $i = 1, \ldots, n$. So, $I_2(Y) : (\prod_j x_j )^\infty = I_2(Y)$ and, consequently, $J \subseteq  I_2(Y)$ as desired. 

Finally, by Proposition \ref{Prop GB1}, we conclude that the $2 \times 2-$minors of $Y$ form a minimal system generators of $J$ and, (\cite{OV2}, Corollary 14), we conclude that $J$ has a unique minimal system of binomial generators.
\end{proof}

We recall that semigroup ideals minimally generated by a Graver basis have unique minimal system of binomials generators (see~(\cite{OV2}, Corollary 16)). As Graver bases are in particular universal Gr\"obner bases (see~\cite{Stu96}, Proposition 4.11), by Example \ref{Ex5}, we can assure the minimal system of binomial generators of $J$ is not a Graver basis. %Therefore, we exhibit a new family of semigroup ideal ideals having unique minimal system of binomial generators.

%%%%%%%%%%%%%%%%%%%%%%%%%%%%%%%%%%%%%%%%%%%%%%%%%%%%%%%%%%%
%%%%%%%%%%%%%%%%%%%%%%%%%%%%%%%%%%%%%%%%%%%%%%%%%%%%%%%%%%%
\section{Gr\"obner Basis and Minimal Generators for \boldmath$I$}\label{S4}

We maintain the notation of the Introduction and the previous Sections, and we set
$\mathcal{G}^{(i)}_4$ to be equal to \[\left\{\underline{x_1^{a+1} x_l^b} - x_{l+1} x_n^b \ \mid\ l = 1, \ldots, i-1  \right\}\bigcup \left\{\underline{x_{l+1} x_n^b} - x_1^{a+1} x_l^b \mid\ l = i, \ldots, n-1  \right\},\] where the underlined monomials again highlight the leading terms with respect to $\prec_i$ of the corresponding binomials.

Let $I_2(X)$ be the ideal of $\Bbbk[x_1, \ldots, x_n]$ generated by the $2 \times 2-$minors of the matrix $X$ defined in \eqref{ecu1}.

\begin{Theorem}\label{Prop GB2}
The set $\mathcal{G}^{(i)} = \mathcal{G}^{(i)}_Y \cup \mathcal{G}^{(i)}_4$ is a minimal Gr\"obner basis of $I_2(X)$ with respect to $\prec_i$. In particular, the cardinality of $\mathcal{G}^{(i)}$ is $\binom{n}2$.
\end{Theorem}

\begin{proof}
Proceeding as in the proof of Proposition \ref{Prop GB1}, we first need to prove that $S(f,g)$ reduces to zero with respect to $\mathcal{G}^{(i)}$, for every $f, g \in \mathcal{G}^{(i)}$. However, as, by Proposition \ref{Prop GB1}, $\mathcal{G}^{(i)}_Y$ is already a Gr\"obner basis with respect to $\prec_i$ and the leading terms with respect to $\prec_i$ of the binomials in $\mathcal{G}^{(i)}_4$ are relatively prime, it suffices to prove that $S(f,g)$ reduces to zero with respect to $\mathcal{G}^{(i)}$, for every $f \in \mathcal{G}^{(i)}_Y$ and $g \in \mathcal{G}^{(i)}_4$. To do this we distinguish three cases:
\begin{itemize}
\item $f \in \mathcal{G}^{(i)}_1 = \big\{ \underline{x_{j+1} x_k^b} - x_j^b x_{k+1} \ \mid\ j \in \{i, \ldots, n-2\},\ k \in \{j+1,\ldots, n-1\}\big\}.$ If $j \neq l$ and $k \neq l+1$, then the leading terms of $f$ and $g$ are relatively prime and there is nothing to prove. Therefore, it suffices to consider the cases $j = l$ or 
$k = l+1$.
\begin{itemize}
\item[$\circ$] If $j = l$, then $l \geq i$; otherwise, the leading terms of $f$ and $g$ are relatively prime, and $S(f,g) = x_n^b (-x_l^b x_{k+1}) - x_k^b (-x_1^{a+1} x_l^b) = -x_l^b( x_{k+1} x_n - x_1^{a+1} x_l^b)$ reduces to zero with respect to $\mathcal{G}^{(i)}_4$. 
\item[$\circ$] If $k = l+1$ then $n-2 \geq k-1 = l \geq j \geq i$, otherwise the leading terms of $f$ and $g$ are relatively prime, and $S(f,g) = x_n^b (-x_j^b x_{l+2}) - x_{j+1} x_{l+1}^{b-1} (-x_1^{a+1} x_l^b) = \underline{- x_j^b x_n^b x_{l+2}} + x_1^{a+1}  x_{j+1} x_{l+1}^{b-1} x_l^b$. Observe that the leading term of $S(f,g)$ is divisible by the leading term of $h := \underline{x_n^b x_{l+2}} - x_1^{a+1} x_{l+1}^b  \in \mathcal{G}^{(i)}_4$. Therefore, $S(f,g) = x_j^b h - x_j^b x_1^{a+1} x_{l+1}^b + x_1^{a+1}  x_{j+1} x_{l+1}^{b-1} x_l^b = x_j^b h - x_1^{a+1} x_{l+1}^{b-1} (x_j^b x_{l+1} - x_{j+1}x_l^b)$. Now, as $x_j^b x_{l+1} - x_{j+1}x_l^b \in \mathcal{G}^{(i)}_Y$, we are done. 
\end{itemize}
\item $f \in \mathcal{G}^{(i)}_2 = \big\{ \underline{x_{j+1} x_k^b} - x_j^b x_{k+1} \ \mid\ j \in \{1, \ldots, i-2\},\ k \in \{j+1,\ldots, i-1\}\big\}$. If $j+1 \neq l$ and $k \neq l$, then the leading terms of $f$ and $g$ are relatively prime and there is nothing to prove. So, it suffices to consider the cases $j = l -1$ or $k = l$.
\begin{itemize}
\item[$\circ$] If $j+1 = l$, then $1 \leq j = l-1 < k \leq i-1$, otherwise the leading terms of $f$ and $g$ are relatively prime, and $S(f,g) = x_1^{a+1} x_l^{b-1} (-x_{l-1}^b x_{k+1}) - x_k^b  (- x_{l+1} x_n^b) = x_k^b x_{l+1} x_n^b - x_1^{a+1} x_{l-1}^b x_l^{b-1} x_{k+1}$. If $l=k$, then the S-polynomial $S(f,g) =  x_k^b x_{k+1} x_n^b$ $- x_1^{a+1} x_{k-1}^b x_k^{b-1} x_{k+1} = - x_k^{b-1} x_{k+1}(\underline{x_1^{a+1} x_{k-1}^b} - x_k x_n^b)$ reduces to zero with respect to $\mathcal{G}^{(i)}_4$; otherwise, the leading term of $S(f,g)$ is $x_k^b x_{l+1} x_n^b$ which is divisible by the leading term of $h := \underline{x_k^b x_{l+1}} - x_{k+1} x_l^b \in \mathcal{G}^{(i)}_4$. 
So, $S(f,g) = x_n^b h + x_n^b x_{k+1} x_l^b  - x_1^{a+1} x_{l-1}^b x_l^{b-1} x_{k+1} = x_n^b h - x_l^{b-1} x_{k+1}(x_1^{a+1} x_{l-1}^b - x_n^b x_l).$ Now, since $x_1^{a+1} x_{l-1}^b - x_n^b x_l \in \mathcal{G}^{(i)}_4$, we are done.
\item[$\circ$] If $k=l$, then $1 \leq j < k = l \leq i-1$, otherwise the leading terms of $f$ and $g$ are relatively prime, and $S(f,g) = x_1^{a+1} (- x_j^b x_{l+1} ) - x_{j+1} (- x_{l+1} x_n^b )= - x_{l+1}( x_1^{a+1} x_j^b - x_{j+1} x_n^b).$ Now, since $x_1^{a+1} x_j^b - x_{j+1} x_n^b \in \mathcal{G}^{(i)}_4$, we are done.
\end{itemize}
\item $f \in \mathcal{G}^{(i)}_3 =  \big\{ \underline{x_j^b x_{k+1}} - x_{j+1} x_k^b \ \mid\ j \in \{1, \ldots, i-1\},\ k \in \{i,\ldots, n-1\}\big\}$. If $j \neq 1, j \neq l, k \neq l$ and $k \neq n-1$, then the leading terms of $f$ and $g$ are relatively prime and there is nothing to prove. Therefore, is suffices so consider the cases $j=1, j = l, k=l$ or $k = n-1$.
\begin{itemize}
\item[$\circ$] If $j=1$, then, in particular, $l < i$, otherwise the leading terms of $f$ and $g$ are relatively prime. Now, if $a+1 \geq b$, then $S(f,g)= \underline{x_1^{a+1-b} x_l^b (- x_2 x_k^b)} - x_{k+1} x_{l+1} x_n^b$ and its leading term is divisible by the leading term of $h := \underline{x_l^b x_2} - x_{l+1} x_1^b \in \mathcal{G}^{(i)}_2 \cup \mathcal{G}^{(i)}_3$; then $S(f,g) = x_1^{a+1-b} x_k^b h - x_1^{a+1-b} x_k^b (- x_{l+1} x_1^b) - x_{k+1} x_{l+1} x_n^b = x_1^{a+1-b} x_k^b h$ $- x_{l+1} (\underline{x_{k+1} x_n^b} - x_1^{a+1} x_k^b)$ which reduces to zero with respect to $\mathcal{G}^{(i)}_2 \cup \mathcal{G}^{(i)}_3 \cup \mathcal{G}^{(i)}_4$. Otherwise, if $a+1 < b$, then $S(f,g) = \underline{x_l^b (-x_2 x_k^b)} - x_1^{b-a-1} x_{k+1} (- x_{l+1} x_n^b)$ and its leading term is divisible by the leading term of $h := \underline{x_l^b x_2} - x_{l+1} x_1^b \in \mathcal{G}^{(i)}_2 \cup \mathcal{G}^{(i)}_3$; then $S(f,g) = x_k^b h$ $- x_k^b (-x_{l+1} x_1^b) - x_1^{b-a-1} x_{k+1}(- x_{l+1} x_n^b)$ $= x_k^b h - x_1^{b-a-1} x_{l+1}$ $(x_{k+1} x_n^b-x_1^{a+1} x_k^b)$ which reduces to zero with respect to $\mathcal{G}^{(i)}_2 \cup \mathcal{G}^{(i)}_3 \cup \mathcal{G}^{(i)}_4$, too.
\item[$\circ$] If $j=l$, then $1 \leq l \leq i-1$; otherwise, the leading terms of $f$ and $g$ are relatively prime, and $S(f,g)= x_1^{a+1} (- x_{l+1} x_k^b) - x_{k+1} (-x_{l+1} x_n^b) = -x_{l+1} (x_{k+1} x_n^b - x_1^{a+1} x_k^b)$ which reduces to zero with respect to $\mathcal{G}^{(i)}_4$.
\item[$\circ$] If $k=l$, then $i \leq l \leq n-1$; otherwise, the leading terms of $f$ and $g$ are relatively prime, and $S(f,g) = x_n^b(-x_{j+1} x_l^b) - x_j^b (-x_1^{a+1} x_l^b) = x_l^b (\underline{x_1^{a+1} x_j^b} - x_{j+1} x_n^b)$ which reduces to zero with respect to $\mathcal{G}^{(i)}_4$.
\item[$\circ$] If $k=n-1$, then $1 \leq j < i \leq l < n$, otherwise the leading terms of $f$ and $g$ are relatively prime. In this case, $S(f,g)= x_{l+1} x_n^{b-1} (-x_{j+1} x_{n-1}^b) \underline{- x_j^b (-x_1^{a+1} x_l^b)}$ and, since the leading term of $S(f,g)$ is divisible by the leading term of $h:=\underline{x_1^{a+1} x_j^b} - x_{j+1} x_n^b \in \mathcal{G}^{(i)}_4$, we have that $S(f,g) = x_l^b h - x_l^b (- x_{j+1} x_n^b) +  x_{l+1} x_n^{b-1} (-x_{j+1} x_{n-1}^b)$ $= x_l^b h - x_{j+1} x_n^{b-1}(\underline{x_{l+1} x_{n-1}^b} - x_l^b x_n)$, and as $\underline{x_{l+1} x_{n-1}^b} - x_l^b x_n$ belongs to $\mathcal{G}^{(i)}_1$, we are done.  
\end{itemize}
\end{itemize}

Now, as $S(f,g)$ reduces to zero with respect to $\mathcal{G}^{(i)}$ in all the three cases we conclude that $\mathcal{G}^{(i)}$ forms a Gr\"obner basis.

Once we know that $\mathcal{G}^{(i)}$ is a Gr\"obner basis, we observe that the leading terms of the binomials in $\mathcal{G}^{(i)}$ are not divisible by the leading term of any other binomial in $\mathcal{G}^{(i)}$ other than itself. That is to say, the Gr\"obner basis $\mathcal{G}^{(i)}$ is minimal.

Clearly, $\mathcal{G}^{(i)}$ is a subset of $2 \times 2-$minors of $X$. Moreover, its cardinality is equal to the cardinality of $\mathcal{G}^{(i)}_Y,$ that is $\binom{n-1}2$, plus the cardinality, $n-1$, of $\mathcal{G}^{(i)}_4$. Therefore, $\mathcal{G}^{(i)}$ has cardinality equal to $\binom{n-1}2 + (n-1) = \binom{n}2$ which is the number of $2 \times 2-$minors of $X$. Hence we conclude that $\mathcal{G}^{(i)}$ generates $I_2(X)$ and we are done.
\end{proof}

\begin{Example}
The minimal Gr\"obner basis, $\mathcal{G}^{(i)}$, of $I_2(X)$ with respect to $\prec_i$ is not reduced in general. For example, if $n = 4, a = 3$ and $b = 3$, then one can see (using, e.g., Singular~\cite{singular}) that the binomial $x_4^4 - x_1 x_2^4 x_3^2 $ belongs to the Gr\"obner basis of $I_2(X)$ with respect to $\prec_2$; however, $x_4^4 - x_1 x_2^4 x_3^2$ is not a minor of $X$.
\end{Example}

\begin{Corollary}\label{Cor GB2}
If $\gcd(a, r_b(n)) = 1$, then the ideal $I$ is minimally generated by the $2 \times 2-$minors of $X$. In this case, if $n > 3$, then $I$ has a unique minimal system of generators if and only if  and $a < b-1$.
\end{Corollary}

\begin{proof}
By Theorem \ref{Prop GB2}, to prove the first part of the statement it suffices to see that $I = I_2(X)$.

By Lemma \ref{lema2}, we have that $\varphi_\mathcal{A}(f) = 0$, for every $f \in \mathcal{G}^{(i)}$, where $\varphi_\mathcal{A}$ is the $\Bbbk-$algebra homomorphism define in \eqref{ecu2}. Therefore $I_2(X) \subseteq I$. Conversely, let $L$ be the $(n-1) \times n-$matrix 
\[\left(\begin{array}{cccccccc} b & -(b+1) & 1 & 0 & \ldots & 0 & 0 & 0 \\ 
0 & b & -(b+1) & 1 & \ldots & 0 & 0 & 0  \\ 
\vdots & \vdots & \vdots & \vdots &  & \vdots & \vdots & \vdots\\ 
0 & 0 & 0 & 0 & \ldots & b & -(b+1) & 1 \\
(a+1) & 0 & 0 & 0 & \ldots & 0 & b & -(b+1) 
\end{array}\right)\] and let $I_L$ be the ideal of $\Bbbk[x_1,  \ldots, x_n]$ generated by \[\left\{\mathbf{x}^{\mathbf{u}_+} - \mathbf{x}^{\mathbf{u}_-} \mid \mathbf{u}\ \text{is a row of}\ L\right\}.\] Clearly, $I_L \subseteq I_2(X)$.

On the one hand, a direct computation shows that the set of $(n-1) \times (n-1)-$minors of 
$L$ is equal (up to sign of its elements) to $\{a_1, \ldots, a_n\}$ and therefore, by~(\cite{Stu96} [Lemma 12.2), \[ 
I_L : \big(\prod_{i=1}^n x_i \big)^\infty = I\] if and only if $\gcd(a_1, \ldots, a_n) = \gcd(a, r_b(n)) = 1$. On the other hand, by Theorem \ref{Prop GB2} and~(\cite
{BSR99}, Theorem 3.1), we have that $I_2(X): \big(x_i^\infty\big) =  I_2(X)$ for every $i = 1, \ldots, 
n$, that is to say, $I_2(X): \big(\prod_{i=1}^n x_i\big)^\infty = I_2(X)$. Putting this 
together we conclude that \[ I_2(X) = I_2(X): \big(\prod_{i=1}^n x_i\big)^\infty \supseteq 
I_L : \big( \prod_{i=1}^n x_i\big)^\infty = I,\]
and thus $I = I_2(X)$ as claimed.

To prove the second part of the statement, we observe that, for every $i \neq n$, the non-leading term, $x_1^{a+1} x_l^b,$ of the binomial $\underline{x_{l+1} x_n^b} - x_1^{a+1} x_l^b  \in \mathcal{G}^{(i)}_4$ is divisible by the leading term of $\underline{x_1^b x_l} - x_2 x_{l-1}^b \in \mathcal{G}^{(i)}_3$, provided that $l \geq 3$ (otherwise, no such binomial in $\mathcal{G}^{(i)}_3$ exists), if and only $a+1 \geq b$. Now, as these are the only divisibility relationships between the monomials of the binomials in $\mathcal{G}^{(i)}$, and $l \geq 3$ implicitly requires $n > 3$ , we obtain that for $n > 3$, $\mathcal{G}^{(i)}$ is reduced for every $\prec_i$ if and only $a < b - 1$, and, by~(\cite{OV2}, Corollary 14), we conclude that for $n > 3,\ I$ has a unique minimal system of binomial generators if and only if and $a < b-1$.
\end{proof}

Notice that the condition $\gcd(a, r_b(n)) = 1$ cannot be avoided.

\begin{Example}
Let $n = 4, a = 3$ and $b = 2$. In this case, $a_1 = r_b(4) = 15, a_2 = 18, a_3 = 24$ and $a_4 = 36$. Clearly, $\gcd(a_1, a_2, a_3, a_4) = \gcd(a,r_b(4)) = 3$. By direct computation, one can check that $I$ is minimally generated by four binomials whereas $I_2(X)$ is minimally generated by six binomials. In particular, $I \neq I_2(X)$; in fact, one has that $I$ is a minimal prime of $I_2(X)$. 
\end{Example}

The following example shows the minimal system of generators of $I$ for $n=4$.

\begin{Example}
If $n=4$, then the ideal $I \subset \Bbbk[x_1,x_2,x_3,x_4]$ is minimally generated by 
\[x_2^{b+1} - x_1^b x_3, x_1^b x_4 - x_2 x_3^b, x_3^{b+1} - x_2^b x_4\]
and
\[x_1^{a+b+1} - x_2 x_4^b, x_1^{a+1} x_2^b - x_3 x_4^b, x_4^{b+1} - x_1^{a+1} x_3^b\]
(recall that the first three binomials generates $J$). 
In~\cite{KaOj}, a complete classification of the monomial curves in $\mathbb{A}^4(\Bbbk)$ having a unique minimal system of generators is given. By~(\cite{KaOj}, Theorem 3.11), one has that $I$ has a unique minimal system of generators if and only if $x_1^{a+1} x_3^b$ is not divisible by $x_1^b x_3$; equivalently $a < b-1$ as we already knew by Corollary \ref{Cor GB2}. Observe that the result on the uniqueness of the system of generators of $I$ can be deduced from~\cite{KT10}, too.
\end{Example}

We end this paper by observing that, since both $J$ and $I$ are determinantal ideals by Corollaries \ref{Cor GB1} and \ref{Cor GB2}, respectively, one can conveniently adapt~(\cite{Philippe}, Section 2.1) to compute the minimal free resolution of $I$ and $J$ using the Eagon--Northcott complex. In particular, one can prove that the $\Bbbk-$algebras $\Bbbk[x_1, \ldots, x_n]/J$ and $\Bbbk[x_1, \ldots, x_n]/I$ are Cohen--Macaulay of type $n-2$ and $n-1$, respectively (see~(\cite{Philippe}, Section 2.1 for further details)). The explicit computation of the minimal free resolution of $\Bbbk[x_1, \ldots, x_n]/J$ and $\Bbbk[x_1, \ldots, x_n]/I$ is a \mbox{future work.}

%%%%%%%%%%%%%%%%%%%%%%%%%%%%%%%%%%%%%%%%%%%%%%%%%%%%%%%%%%%
%%%%%%%%%%%%%%%%%%%%%%%%%%%%%%%%%%%%%%%%%%%%%%%%%%%%%%%%%%%

\section*{Author Contributions}

All authors contributed equally. All authors have read and agreed to the published version of the manuscript.

\section*{Funding}

This work was supported in part by Grant MTM2017-84890-P funded by MCIN/FEDER-UE, by Grant PGC2018-096446-B-C21 funded by MCIN/AEI/10.13039/501100011033, by "ERDF A way of making Europe", and by Grant GR21055(FQM-024) funded by Junta de Extremadura (Spain)/FEDER funds.

\section*{Acknowledgments}

The authors wish to thank the referees for their valuable comments and suggestions that have undoubtedly helped to increase the clarity of the presentation of this paper.

\section*{Conflicts of Interest}

The authors declare no conflicts of interest.

\end{document}